\documentclass[12pt,reqno]{amsart}

\usepackage{euscript}
\usepackage{amssymb}
\usepackage{amsthm}
\usepackage{amsmath}
\usepackage{a4wide}
\usepackage{latexsym}
\usepackage{mathrsfs}
\usepackage[OT1]{fontenc}
\usepackage[latin1]{inputenc}

\usepackage{verbatim}

\DeclareFontFamily{U}{mathx}{\hyphenchar\font45}
\DeclareFontShape{U}{mathx}{m}{n}{
      <5> <6> <7> <8> <9> <10>
      <10.95> <12> <14.4> <17.28> <20.74> <24.88>
      mathx10
      }{}
\DeclareSymbolFont{mathx}{U}{mathx}{m}{n}
\DeclareFontSubstitution{U}{mathx}{m}{n}
\DeclareMathAccent{\widecheck}{0}{mathx}{"71}
\DeclareMathAccent{\wideparen}{0}{mathx}{"75}

\usepackage{amssymb}
\usepackage{mathrsfs}

\newtheorem{theorem}{Theorem}[section]

\newtheorem{proposition}[theorem]{Proposition}
\newtheorem{corollary}[theorem]{Corollary}

\theoremstyle{definition}

\numberwithin{equation}{section}

\newcommand{\eten}{\,\widecheck{\otimes}\,}

\newcounter{smallromans}

\newenvironment{romanenumerate}
{\begin{list}{{\normalfont\textrm{(\roman{smallromans})}}}%
    {\usecounter{smallromans}\setlength{\itemindent}{0cm}%
      \setlength{\leftmargin}{5.5ex}\setlength{\labelwidth}{5.5ex}%
      \setlength{\topsep}{0.2ex}\setlength{\partopsep}{0ex}%
      \setlength{\itemsep}{0.2ex}}}%
  {\end{list}}

\newcounter{smallalphs}

  {\end{list}}

\renewcommand{\phi}{\ensuremath{\varphi}}
\renewcommand{\epsilon}{\ensuremath{\varepsilon}}

\begin{document}
\title[On C*-algebras which cannot be decomposed into tensor products]{On C*-algebras which cannot be decomposed\\ into tensor products with both factors infinite-dimensional}


\author[T.~Kania]{Tomasz Kania}
\address{Department of Mathematics and Statistics, Fylde College, Lancaster University, Lancaster LA1 4YF, United Kingdom.}
\email{tomasz.marcin.kania@gmail.com}
\maketitle
\begin{abstract}We prove that C*-algebras which, as Banach spaces, are Grothendieck cannot be decomposed into a tensor product of two infinite-dimensional C*-algebras. By a result of Pfitzner, this class contains all von Neumann algebras and their norm-quotients. We thus complement a recent result of Ghasemi who established a similar conclusion for the class of SAW*-algebras. \end{abstract}
\section{Introduction and statement of the main result}

During the London Mathematical Society Meeting held in Nottingham on $6^{{\rm th}}$ September 2010, Simon Wassermann asked a question of whether the Calkin algebra can be decomposed into a C*-tensor product of two infinite-dimensional C*-algebras. This question stems from the study of the elusive nature of the automorphism group of the Calkin algebra whose structure is independent of the usual axioms of Set Theory (\cite{farah,phillips}). Ghasemi (\cite{gha}) studied tensorial decompositions of SAW*-algebras answering the above-mentioned question in the negative---one cannot thus expect to build automorphisms of such algebras out of automorphisms of non-trivial tensorial factors. Let us remark that the commutative version of Ghasemi's result was known to experts as it follows directly from the conjunction of \cite[Theorem B]{seever} with the main theorem of \cite{cem}. \medskip

The aim of this note is to prove that C*-algebras which satisfy a certain Banach-space property cannot be decomposed into a tensor product of C*-algebras. More specifically, we prove that C*-algebras which, as Banach spaces, are Grothendieck (\emph{i.e.}, weak*-null sequences in the dual space converge weakly) do not allow such a tensorial decomposition. In particular, we give a new solution to the problem of Wassermann as the Calkin algebra falls into the class of Grothendieck spaces.

\begin{theorem}\label{main}Let $A$ be a C*-algebra which, as a Banach space, is a Grothendieck space. Suppose that $E$ and $F$ are C*-algebras such that $$A\cong E\otimes_\gamma F$$ for some C*-norm $\gamma$. Then either $E$ or $F$ (or both) are finite-dimensional. \smallskip

In other words, a C*-algebra, which is also a Grothendieck space, cannot be decomposed into a tensor product of two infinite-dimensional C*-algebras.\end{theorem}
\pagebreak
Let us list examples of classes of C*-algebras which meet the assumptions of Theorem~\ref{main}.
\begin{proposition}\label{classes}C*-algebras in each of the following classes are Grothendieck spaces:
\begin{romanenumerate}
\item von Neumann algebras (or more generally, AW*-algebras) and their norm-quotients; in particular $\mathscr{B}(H)$ and the Calkin algebra,
\item ultraproducts of C*-algebras over countably incomplete ultrafilters,
\item \label{sepinj}unital C*-algebras with the countable Riesz interpolation property,
\end{romanenumerate} \end{proposition}

\begin{proof}By Corollary~\ref{vna}, von Neumann algebras and hence their continuous, linear images (such as the Calkin algebra) satisfy the hypothesis of Theorem~\ref{main}. The assertion for AW*-algebras follows from Proposition~\ref{abelian} as every maximal abelian self-adjoint subalgebra of an AW*-algebra is Grothendieck by the main results of \cite{sm1} and \cite{sm}.\medskip

Avil\'{e}s \emph{et al.}~(\cite[Proposition~3.3]{av}) proved that ultraproducts of Banach spaces over countably incomplete ultrafilters cannot contain complemented copies of $c_0$. This, combined with Theorem~\ref{pfi} and Proposition~\ref{grchar}\eqref{g4}, yields that ultraproducts of C*-algebras are Grothendieck spaces.\medskip

The assertion \eqref{sepinj} follows from applying \cite[Theorem 9]{pol} to the real Banach space $A_{{\rm sa}}$ of all self-adjoint elements of a unital C*-algebra $A$ with the countable Riesz interpolation property and noticing that the Grothendieck property passes from $A_{{\rm sa}}$ to the complex Banach space $A = A_{{\rm sa}}\oplus iA_{{\rm sa}}$.
\end{proof} 

It is perhaps worthwhile to mention that even at the abelian level there exist many Grothendieck $C(X)$-spaces that are not SAW* (\emph{i.e.}, for which $X$ is not sub-Stonean); an example of such is a space constructed by Haydon (\cite{haydon}).\medskip

Each unital C*-algebra with the countable Riesz interpolation property is an SAW*-algebra in the sense of Pedersen (\cite[Proposition 2.7]{smith}); see \cite{ped} for the definition of an SAW*-algebra. Conjecturally all unital SAW*-algebras have the countable Riesz interpolation property (\cite[p.~117]{smith}), hence our result covers all known examples of SAW*-algebras---thus, we extend the main result of \cite{gha} to the class of ultraproducts of C*-algebras and other C*-algebras created, for instance, out of Grothendieck abelian C*-algebras that are not SAW*.\medskip 

To the best of our knowledge, it is not known whether a maximal abelian self-adjoint subalgebra of an SAW*-algebra is SAW* too. If this were the case, Proposition~\ref{abelian} would immediately imply that SAW*-algebras are Grothendieck spaces, because abelian SAW*-algebras are of the form $C_0(X)$ for some locally compact sub-Stonean space $X$, hence Grothendieck by \cite{ando} or \cite{seever}.

\section{Preliminaries}

\subsection*{Grothendieck spaces}

A series $\sum_{n=1}^\infty y_n$ in a Banach space $E$ is \emph{weakly unconditionally convergent} if the scalar series $\sum_{n=1}^\infty |\langle f, y_n\rangle|$ converges for each $f\in E^*$. An operator between Banach spaces is \emph{unconditionally converging} if it maps weakly unconditionally convergent series to unconditionally convergent series. A Banach space $E$ has \emph{property (V)} if for each Banach space $F$ the class of unconditionally converging operators $T\colon E\to F$ coincides with the class of weakly compact operators. It is a result of Pe\l czy\'{n}ski that $C(X)$-spaces (abelian C*-algebras) have property (V) (\cite{pel}).\medskip

A Banach space $E$ is \emph{Grothendieck} if every weak*-null sequence in $E^*$ converges weakly. The name \emph{Grothendieck space} stems from a result of Grothendieck (\cite{gr}) who identified $\ell_\infty$ as a space having this property. It is a trivial remark that reflexive spaces have this property too. By the Hahn--Banach theorem, the class of Grothendieck spaces is closed under surjective linear images, \emph{i.e.}, whenever $E$ and $F$ are Banach spaces, $T\colon E\to F$ is a surjective bounded linear operator, then if $E$ is Grothendieck, so is $F$. \medskip

Let us record a proposition which links property (V) with the Grothendieck property.

\begin{proposition}\label{grchar}Let $X$ be a Banach space. Then the following are equivalent.
\begin{romanenumerate}
\item \label{g1}$X$ is a Grothendieck space,
\item \label{g3}each bounded linear operator $T\colon X\to c_0$ is weakly compact.
\item \label{g4}$X$ has property (V) and no subspace of $X$ isomorphic to $c_0$ is complemented.
\end{romanenumerate}
\begin{proof}For the proof of equivalences \eqref{g1} $\iff$ \eqref{g3} see \cite[Corollary 5 on p.~150]{diestel2}. Equivalence \eqref{g1} $ \iff $ \eqref{g4} is due to R\"{a}biger (\cite{rab}); see also \cite[Theorem 28]{ghelew} (this argument is also implicit in the proof of \cite[Corollary 2]{cem}).\end{proof}

\end{proposition}

We require the following theorem of Pfitzner (\cite[Theorem 1]{pfi}), which can be thought as a non-commutative generalisation of the above-mentioned result of Pe\l czy\'{n}ski.

\begin{theorem}[Pfitzner]\label{pfi}Let $A$ be a C*-algebra and let $\EuScript{K}\subset A^*$ be a bounded set. Then $\EuScript{K}$ is not relatively weakly compact if and only if there are a sequence $(x_n)_{n=1}^\infty$ of pairwise orthogonal, norm-one self-adjoint elements in $A$ and $\delta>0$ such that
$$\sup_{f\in \EuScript{K}} | \langle f, x_n \rangle | >\delta.$$
In particular, C*-algebras have property (V). \end{theorem}

The original proof was highly sophisticated and relied on numerous deep facts from Banach space theory. Fortunately, Fern\'{a}ndez-Polo and Peralta (\cite{pp}) supplied a short and elementary proof of Pfitzner's theorem. \medskip

By virtue of Proposition~\ref{grchar}\eqref{g4}, we arrive at the following corollary.
\begin{corollary}\label{gengroth}A C*-algebra is a Grothendieck space if and only if it does not contain complemented subspaces isomorphic to $c_0$. \end{corollary}
The special case of Corollary~\ref{gengroth} where the C*-algebra is also a von Neumann algebra was noted by Pfitzner (\cite[Corollary 7]{pfi}):
\begin{corollary}\label{vna}Von Neumann algebras are Grothendieck spaces.\end{corollary}
Let us take this opportunity to record the following easy corollary to Theorem~\ref{pfi}.
\begin{proposition}\label{abelian}Let $A$ be a C*-algebra with the property that each maximal abelian self-adjoint subalgebra $B$ of $A$ is a Grothendieck space. Then $A$ is a Grothendieck space.  \end{proposition}
\begin{proof}Let $T\colon A\to c_0$ be a bounded linear operator. By Proposition~\ref{grchar}\eqref{g3} it is enough to show that $T$ is weakly compact. 

Assume contrapositively that $T$ is not weakly compact. By Gantmacher's theorem, $T$ is weakly compact if and only if $T^*$ is, so the set $\EuScript{K} = T^*(B)$ is not relatively weakly compact, where $B$ is the unit ball of $c_0^*$. By Theorem~\ref{pfi}, there exist $\delta>0$ and a sequence $(x_n)_{n=1}^\infty$ of pairwise orthogonal, norm-one self adjoint elements in $A$ such that 
\begin{equation}\label{wc}\sup_{f\in \EuScript{K}} | \langle f, x_n \rangle | = \sup_{y\in B} | \langle T^*y, x_n \rangle |  = \sup_{y\in B} | \langle y, Tx_n \rangle | >\delta.\end{equation}
Let $B_0\subseteq A$ be the C*-algebra generated by $\{x_n\colon n\in \mathbb{N}\}$. Since the $x_n$ ($n\in \mathbb{N}$) are pairwise orthogonal, $B_0$ is abelian. Let $B$ be a maximal abelian subalgebra of $A$ containing $B_0$. Consider the restriction $T|_B\colon B\to c_0$. It is not weakly compact by \eqref{wc}, so $B$ is not a Grothendieck space. \end{proof}

\section{Proof of Theorem~\ref{main}}
We are now in a position to prove our main result. The general strategy of the proof was inspired by the path taken by Cembranos in \cite{cem}.
\begin{proof}[Proof of Theorem~\ref{main}]Let $A$ be a C*-algebra and suppose that it is a Grothendieck space. Assume towards a contradiction that $A\cong E\otimes_\gamma F$ for some infinite-dimensional C*-algebras $E,F$ and a C*-norm $\gamma$.\medskip

Since *-homomorphisms between C*-algebras have closed range, there is a surjective *-homomorphism $Q\colon E\otimes_\gamma F\to E \otimes_{{\rm min}} F$ that extends the identity map on $E\odot F$. (Here $E \otimes_{{\rm min}} F$ denotes the minimal C*-tensor product of $E$ and $F$.)\medskip

Let $B_1\subset E$ and $B_2\subset F$ be infinite-dimensional abelian C*-algebras. (Such subalgebras exist because infinite-dimensional C*-algebras contain self-adjoint elements with infinite spectrum (\cite[Ex.~4.6.12]{kari}), hence the assertion follows from the spectral theorem.) We may thus identify $B_1 \otimes_{\rm min} B_2$ with a subalgebra of $E\otimes_{\rm min} F$ (\emph{cf}. \cite[II.9.6.2]{black}). However, the (minimal) tensor product of abelian C*-algebras is the same as the Banach-space injective tensor product, \emph{i.e.}, $$B_1 \otimes_{\rm min} B_2 = B_1 \eten B_2.$$

Let $(e_n)_{n=1}^\infty$ be a sequence of pairwise orthogonal, positive, norm-one elements in $B_1$. In particular, $E_0 = \overline{{\rm span}}\{e_n\colon n\in \mathbb{N}\}$ is isometric to $c_0$ and $(e_n)_{n=1}^\infty$ is equivalent to the canonical basis for $c_0$. Choose norm-one functionals $e_n^* \in E^*$ such that $\langle e_n^*, e_m\rangle = \delta_{n,m}$ ($n,m\in \mathbb{N}$). Moreover, let $(x^*_n)_{n=1}^\infty$ be a sequence of unit vectors in $F^*$ which converges to 0 in the weak* topology. (Such a sequence exists by the Josefson--Nissenzweig theorem (see \cite[Chapter XII]{diestel}.) Let $(x_n)_{n=1}^\infty\subset F$ be a sequence such that $\langle x_n, x_n^*\rangle =1$. Without loss of generality we may suppose that $\|x_n\|\leqslant 2$ for all $n$.\medskip

Define a map $T\colon E\odot_{\rm min} F \to \ell_\infty$ by the formula
\[ T\xi = ( \langle e_n^* \otimes x_n^*, \xi\rangle )_{n=1}^\infty\qquad (\xi \in E\odot_{\rm min} F).\]
This is a well-defined bounded linear operator because $( e_n^* \otimes x_n^*)_{n=1}^\infty$ is a bounded sequence of functionals on $E\odot_{\rm min} F$. We can thus extend $T$ to the whole of $E\otimes_{\rm min} F$. Moreover, for all $f\in E$ and $x\in F$ we have
$$| \langle e_n^*, f \rangle \cdot  \langle x, x_n^*\rangle | \leqslant \|f\| \cdot |\langle x, x_n^*\rangle|  $$ 
so $T$ takes values in $c_0$ as $(x^*_n)_{n=1}^\infty$ is a weak*-null sequence. 
Since the injective tensor product `respects subspaces' (see \cite[p.~49]{ryan}), $E_0 \eten B_2$ can be identified with a subspace of $B_1\eten B_2$ and the latter is a subspace of $E\otimes_{\rm min} F$. \medskip

As $E_0$ and $c_0$ are isometrically isomorphic, so are $E_0\eten B_2$ and $c_0\eten B_2 \cong c_0(B_2)$ (\emph{cf}.~\cite[Example 3.3]{ryan}).
Let $n\in \mathbb{N}$ and let $(a_k)_{k=1}^\infty$ be a scalar sequence with only finitely many non-zero entries. We have
\[\big\|\sum_{k=1}^n a_k e_k \otimes x_k\big\| = \big\|\sum_{k=1}^n e_k \otimes (a_k x_k)\big\| \leqslant \sup_{1\leqslant k \leqslant n} \|a_k x_k\|  \leqslant 2\max\{ |a_k|\colon 1\leqslant k \leqslant n\}.\]
By \cite[Proposition 4.3.9]{meg}, $\sum_{n=1}^\infty e_n \otimes x_n$ is a weakly unconditionally convergent series in $E\otimes_{{\rm min}} F$. On the other hand, for all $k,n\in \mathbb{N}$ we have $(T(e_n \otimes x_n))(k) =\delta_{k,n}$ so $$\sum_{n=1}^\infty T(e_n \otimes x_n)$$
fails to converge in $c_0$. Consequently, $E\otimes_{{\rm min}} F$ is not a Grothendieck space as we proved that the $c_0$-valued operator $T$ is not unconditionally converging. Indeed, if $E\otimes_{{\rm min}} F$ were Grothendieck, $T$ would be weakly compact (Proposition~\ref{grchar}\eqref{g3}), hence also unconditionally converging (Proposition~\ref{grchar}\eqref{g4}).\end{proof}

\subsection*{Acknowledgments} We are indebted to Y.~Choi and N.~Ozawa for spotting an error in the previous version of the proof of Theorem~\ref{main}.

\bibliographystyle{amsplain}

\end{document}